\documentclass{amsart}
\usepackage{amsmath}
\usepackage{amssymb}
\usepackage{amscd}
\usepackage{amsrefs}
\usepackage{latexsym}
\usepackage{graphicx} 
\usepackage{pinlabel}

\usepackage{color}

\theoremstyle{plain}
\newtheorem{theorem}{Theorem}[section]
\newtheorem{lemma}[theorem]{Lemma}

\newtheorem{proposition}[theorem]{Proposition}
\newtheorem*{question}{Question}
\newtheorem*{definition}{Definition}

\theoremstyle{definition}

\theoremstyle{remark}
\newtheorem{remark}{Remark}

\newcommand{\T}{\mathcal{T}_K}

\newcommand{\TTj}{\mathcal{T}_{K_j}}
\newcommand{\spinc}{$spin^{\text{c}}$ }

\begin{document}

\baselineskip.5cm
\title {Null-homologous exotic surfaces in 4-manifolds}

\author{Neil R. Hoffman}
\address{Department of Mathematics\newline
\hspace*{.375in} Oklahoma State University \newline
\hspace*{.375in} 401 Mathematical Sciences Building, \newline
\hspace*{.375in} Stillwater, OK 74078-1058}
\email{\rm{neil.r.hoffman@okstate.edu}}

\author[Nathan S. Sunukjian]{Nathan S. Sunukjian}
\address{Department of Mathematics and Statistics \newline
\hspace*{.375in} North Hall \newline
\hspace*{.375in} Calvin College \newline
\hspace*{.375in} 3201 Burton SE \newline
\hspace*{.375in} Grand Rapids, MI 49546}
\email{\rm{nss9@calvin.edu}}

\newcommand{\neil}[1]{{\color{blue}{#1}}}
\newcommand{\nathan}[1]{{\color{red}{#1}}}
\newcommand{\referee}[1]{{\color{cyan}{#1}}}

\begin{abstract}

In this paper we exhibit infinite families of embedded tori in 4-manifolds that are topologically isotopic but smoothly distinct. The interesting thing about these tori is that they are topologically trivial in the sense that each bounds a topologically embedded solid handlebody. This implies that there are stably ribbon surfaces in 4-manifolds that are not ribbon. 
 \end{abstract}
\maketitle

\section{introduction}

Just as a 4-manifold can have many inequivalent smooth structures, there can be many different smooth embeddings of surfaces into a 4-manifold which are topologically isotopic, but smoothly distinct. Any surface which admits more than one smooth embedding in a topological isotopy class will be said to admit \emph{exotic embeddings}. 

In this paper we will show that null-homologous tori first discovered by Fintushel and Stern in their knot surgery construction in fact provide examples of  exotically embedded tori. Specifically, 

\begin{theorem}\label{t:main}
Let $X$ be a smooth 4-manifold with $b_2 \geq |\sigma | + 6$, non-trivial Seiberg-Witten invariant, and embedded torus $T$ of self intersection 0 such that $\pi_1(X \setminus T) = 1$. Then $X$ contains an infinite family of distinct tori $\{T_i\}$ that are topologically isotopic to the unknotted torus (a torus that bounds a solid handlebody in $X$), but no diffeomorphism of $X$ exists taking $T_i$ to $T_j$ if $i\ne j$.
\end{theorem}

 The first examples of orientable exotic embeddings come from Fintushel and Stern's ``rim surgery'' technique \cite{FSsurf}. Their surfaces all have simply connected complement. A variation on rim surgery was given by Kim, and Kim-Ruberman which works in the case that the complement has non-trivial fundamental group (\cite{Kim, RK, RK2}). Tom Mark has used Heegaard-Floer homology to show that these constructions are also effective for constructing exotic embeddings of surfaces with negative self intersection (\cite{T}). On the other hand, all of these constructions involve surfaces whose complement has finite first homology, and moreover all of these constructions essentially begin with symplectically embedded surfaces in a symplectic 4-manifold. Such surfaces can never be null-homologous. The significance of our examples is that they are null-homologous and moreover topologically trivial.
 
One of the features of this theorem is that there are numerous tractable examples where the theorem can be applied. For example, any elliptic surface contains such a torus and has non-trivial Seiberg-Witten invariant by virtue of being a symplectic manifold. 

The strategy of proof is as follows: The knot surgery construction of Fintushel and Stern produces an infinite family of exotic smooth structures on a 4-manifold through a series of log-transforms on null-homologous tori. These are the tori we will focus on. Using Seiberg-Witten theory, we will define a gauge theoretic invariant of null-homologous tori to distinguish the tori smoothly. Finally, we will show that all such tori are topologically isotopic by a theorem of the second author:

\begin{theorem}[{\cite[Theorem 7.2]{N}}] \label{t:iso} Let $\Sigma_0$ and $\Sigma_1$ be locally flat embedded surfaces of the same genus in a simply connected 4-manifold $X$. The surfaces are topologically isotopic when $\pi_1(X \setminus \Sigma_i) = \mathbb{Z}$ and $b_2 \geq |\sigma| + 6$.
\end{theorem}
Note that a trivially embedded surface in any 4-manifolds will satisfy $\pi_1(X \setminus \Sigma) = \mathbb{Z}$.

One might wonder how robust these exotic embeddings are. That is, what does it take to make any of the exotically embedded topologically trivial surfaces constructed here smoothly equivalent again? In \cite{BS}, Inanc Baykur and the second author show that these tori become smoothly equivalent once one increases the genus of each of these surfaces in the most trivial possible way. Namely, tubing any one of the topologically trivial surfaces of Theorem 1.1 to a smoothly trivial torus results in a smoothly trivial surface. 

It would be interesting to know what the simplest examples of exotic embeddings are. For example, this paper provides context for the following two natural questions, which have motivated work on exotically embedded surfaces.

\begin{question}
Do there exist exotically embedded surfaces in $S^4$? In particular, is there an embedded $S^2$ that is topologically isotopic to the unknot but not smoothly isotopic to the unknot?
\end{question}

The examples in this paper can be seen as prototypes for answering this sort of question,  as answering this question in other manifolds can be seen as a first step to better understanding the $S^4$ question.  At the same time, the examples of this paper can better inform an attack on the question above. 

We conclude the introduction with the following remark:

\begin{remark}
The exotically embedded tori constructed in this paper provide examples of stably ribbon surfaces that are not ribbon. They are stably ribbon, because they stably smoothly trivial (by \cite{BS}). But they are not ribbon because as we shall see, log transforms on these tori have a different effect than log transforms on trivial tori, whereas \cite[Theorem 8.3]{N} implies that log transforms on ribbon tori are equivalent to those on trivial tori.   
\end{remark}

\textbf{Acknowledgements:} Both authors would like to thank the Max Planck Institute for Mathematics for hosting them while they worked on this project, and Danny Ruberman and Tom Mark for their comments on an early draft of this paper. We are also especially grateful to Inanc Baykur who provided extensive comments on a later version. The first author was partially supported by grant from the Simons Foundation (\#524123 to Neil R. Hoffman) during later revisions of this paper.

\section{Constructing the tori}\label{s:constructing}

Let $T$ be an embedded torus with self intersection zero in a 4-manifold $X$ such that $\pi_1(X \setminus T) = 1$.  We will not construct exotic embeddings of $T$, (any such torus is necessarily homologically essential since it will have a dual), but rather we will find exotic embeddings of nearby null-homologous tori which arise in the ``knot surgery'' construction of Fintushel and Stern (\cite{FSknot} and \cite{Fknot}). Knot surgery along torus $T$ using a knot $K\subset S^3$ is most straightforwardly defined as $X_K = (X \setminus \nu (T)) \cup (S^1\times S^3\setminus \nu (K))$ where the union is formed by taking the longitude of $K$ to the meridian of $T$ (apart from this requirement, the gluing is not, strictly speaking, well defined, and $X_K$ may depend on the gluing map, but this ambiguity does not factor into our argument below).
Fintushel and Stern proved that $X$ is homeomorphic to $X_K$ under the assumption that the complement of $T$ is simply connected, and they further proved that their Seiberg-Witten invariants are related by $SW_{X_K} = SW_X \cdotp \Delta_K(2[T])$ where $\Delta_K$ is the Alexander polynomial for $K$. Therefore, by varying $K$, one can construct infinitely many smooth structures on $X$. 

The Seiberg-Witten formula is proved by viewing knot surgery as a series of log-transforms on null-homologous tori. That is, rather than cutting out $\nu(T) = S^1 \times (S^1 \times D^2)$ and replacing it with $S^1\times S^3\setminus \nu (K)$, we can view knot surgery as a series of log-transforms in $S^1 \times (S^1 \times D^2)$ which eventually lead to $S^1\times S^3\setminus \nu (K)$. Forgetting the extra $S^1$ direction for the moment, one can go from $S^3\setminus \nu (K)$ to $(S^1 \times D^2)$, the complement of the unknot, by doing $\pm 1$ surgery along crossings of $K$ to unknot it. Crossing this whole picture with $S^1$ gives the log-transforms needed for knot surgery. It is these null-homologous tori that are needed to do the knot surgery that will turn out to be the exotically unknotted tori that we are looking for in our theorem.

\begin{figure}
\labellist
\small\hair 2pt
\pinlabel{$\gamma_K$} at 73 101
\pinlabel {{$\times S^1$}} at 153 25
\pinlabel {{$\times S^1$}} at 377 25
\pinlabel {$K$} at 232 130
\endlabellist	
\includegraphics[height=1.75in]{knotsurgery2}
	\caption{\label{f:knotsurgery} The left figure represents $\nu T$, which can be thought of as the complement of the unknot in $S^3$ crossed with $S^1$. Performing a (+1)-log transform on $\T = S^1\times \gamma_K$ gives $S^1\times (S^3\setminus \nu K)$, depicted on the right.}
	
	\end{figure}

We'll focus on the following particular situation: Suppose that $K$ is a knot of unknotting number 1, and $T$ is a torus in $X$ with trivial normal bundle and simply connected complement. Then knot surgery is the result of doing a single log transform on the null homologous torus $\T = S^1 \times \gamma_K$ in $\nu T \subset X$ (see Figure \ref{f:knotsurgery}). The proof of our theorem will show that $\T$ is topologically unknotted, but smoothly non-trivial. To determine the topological isotopy class of $\T$ we will now compute the fundamental group of $X \setminus \T$: Note that since $\pi_1(X\setminus T) = 1$, we have that the inclusion $\pi_1(\partial\nu T) \rightarrow \pi_1(X\setminus \nu T)$ is trivial. Therefore, by repeated applications of Van-Kampen's theorem, 

\begin{equation*}
\begin{aligned}
\pi_1(X \setminus \T) &= \pi_1((X\setminus \nu T) \cup (\nu T \setminus \T) \\
&=\frac{\pi_1(\nu T \setminus  \T)}{\pi_1(\partial \nu T )}\\
&=\frac{\pi_1(S^1 \times (S^3 \setminus (\nu U \cup \gamma_K)))}{\langle S^1\rangle \times \pi_1(\partial \nu U)} \\
&= \pi_1(S^3 \setminus \gamma_k) = \mathbb{Z}
\end{aligned}
\end{equation*}

The second to last line is just a change of notation, since $\nu T \setminus  \T$ is just the same thing as $S^1 \times (S^3 \setminus (\nu U \cup \gamma_K))$ where $U$ is the unknot (see Figure \ref{f:knotsurgery}). And the final equality is true because $\gamma_k$ is necessarily unknotted in $S^3$.

Already we see that this gives at least one exotically embedded torus. Specifically, $\T$ is topologically standard by Theorem \ref{t:iso}, and moreover, performing a log-transform on $\T$ will give an exotic smooth structure on $X$, whereas performing a log-transform on the standardly embedded torus, (i.e. the one that bounds a solid handlebody), will not. Therefore these tori are smoothly distinct, but by Theorem \ref{t:iso} they must be topologically isotopic.

To construct infinite families of exotic surfaces, we need to be more careful. For example suppose $K_1$ and $K_2$ are knots with associated null homologous tori $\mathcal{T}_{K_1}$ and $\mathcal{T}_{K_2}$. Then it is conceivable that one might be able to construct both $X_{K_1}$ and $X_{K_2}$ by some surgery on $\mathcal{T}_{K_1}$. In this circumstance we would not be able to distinguish  $\mathcal{T}_{K_1}$ from $\mathcal{T}_{K_2}$ as we did above. To resolve this issue, we have to look more deeply at how the Seiberg-Witten invariant changes under log-transforms on $\T$, and restrict ourselves to certain classes of knots.

\section{Smooth invariants of null-homologous tori}
The Seiberg-Witten invariant of a 4-manifold $X$ is a map $SW_X : \mathcal{S}\longrightarrow \mathbb{Z}$, where $\mathcal{S}$ is the set of isomorphism classes of \spinc structures on $X$. The \emph{basic classes} of $X$ are defined to be the \spinc structures that map to non-zero integers. It is a well known property of the Seiberg-Witten invariant that a closed 4-manifold has only a finite number of basic classes. Below, we will often not distinguish between a \spinc structure and its first Chern class or even the Poincare dual of its first Chern class.

We will distinguish our null-homologous tori by computing an invariant that is, in a technical sense clarified below, related to the Seiberg-Witten basic classes of the complement of the tori. To do this we will need to understand how the Seiberg-Witten invariant of a 4-manifold is affected by log-transforms. Suppose we are given a 4-manifold with $T^3$ boundary, e.g. $X\setminus \nu T$, and suppose $H_1(T^3) = \mathbb{Z}[a,b,c]$. Denote the log-transformed 4-manifold constructed by gluing on a $D^2\times T^2$, where $[\partial D^2]$ is glued to $[pa+qb+rc]$ as $X_T{(p,q,r)}$ or sometimes just $X_{(p,q,r)}$, and denote the core torus in the $D^2\times T^2$ part of this manifold as $T_{(p,q,r)}$.

A formula of Morgan-Mrowka-Szabo from \cite{MMS} give a formula relating the Seiberg-Witten invariants of various log-transforms:

\begin{align}  \sum_i & SW_{X_{T}(p,q,r)}(k_{(p,q,r)} + i[T_{(p,q,r)}]) = p\sum_i  SW_{X_{T}(1,0,0)}(k_{(1,0,0)} + i[T_{(1,0,0)}]) \\
 & +q\sum_i SW_{X_{T}(0,1,0)}(k_{(0,1,0)} + i[T_{(0,1,0)}]) + r\sum_i SW_{X_{T}(0,0,1)}(k_{(0,0,1)} + i[T_{(0,0,1)}]) \nonumber
 \label{e:eq1}
\end{align}

\noindent where the \spinc structures agree away from the log-transformed tori, i.e.:

\begin{align} k_{(p,q,r)}|_{X_{(p,q,r)}\setminus T_{(p,q,r)}} &= k_{(1,0,0)}|_{X_{(1,0,0)}\setminus T_{(1,0,0)}}\nonumber \\
&= k_{(0,1,0)}|_{X_{(0,1,0)}\setminus T_{(0,1,0)}} \\
&= k_{(0,0,1)}|_{X_{(0,0,1)}\setminus T_{(0,0,1)}}\nonumber. 
\end{align}

Suppose that $T_{(p,q,r)}$ is null-homologous. Then the left hand side of Equation 1 has only one term. Moreover, since $X_{(1,0,0)}$ has only a finite number of basic classes $k_{(1,0,0)} + i[T_{(1,0,0)}]$, we see that Equation 2 implies that there are only a fixed finite number of possible basic classes $k_{(p,q,r)}$ for $X_{T(p,q,r)}$, and these possibilities depend only on the choice of $T$, not on $(p,q,r)$. To put this another way, there are only a finite number of \spinc structures on $X\setminus \nu T$ that can be extended to basic classes on $X_T(p,q,r)$ when $[T_{(p,q,r)}]=0$. Therefore, the following invariant is well defined:

\begin{definition}
Let $T$ be a null-homologous torus in $X$ (defined up to smooth isotopy). Define $B(X,T)$ to be the maximum divisibility of the difference between any two basic classes of {$X_T{(p,q,r)}$}  
for any $(p,q,r)$ such that $[T_{(p,q,r)}] = 0$.

\begin{remark}
It is important to use the divisibility of the \emph{difference} of basic classes rather than just the divisibility of the basic classes because very often after performing knot surgery all of the basic classes have divisibility 1.
\end{remark}
\end{definition}

\section{Families of unknotting number one knots, and the proof of Theorem 1}
Now that we have a better understanding of the smooth invariants needed to distinguish potential infinite families of smooth tori, we can describe an explicit family of knots that will give rise to smoothly distinct $\T$. For the invariant $B(X,T)$ to be useful, we will need to find a family of knots with unknotting number $1$ whose Alexander polynomials have arbitrarily high degree. All two-bridge knots can be given in the form of Figure \ref{f:twobridge} where $a_i$ is the number of right half-twists when $i$ is odd, and left half-twists when $i$ is even. We refer to two-bridge knots using Conway's notation, $C(a_0,\ldots,a_m)$, and we note that it is well known (see \cite[$\S$12B]{BZ} for instance), that two 2-bridge knots are equivalent if and only if $[a_0,\ldots,a_m]$ and $[a'_0,\ldots,a'_{m'}]$ are continued fraction expansions of the same rational number.

\begin{proposition}[Kanenobu-Murakami \cite{KM}]\label{p:unknotprop}
A two-bridge knot has unknotting number one if and only if it can be expressed as 
$$C(b,b_1,b_2,\ldots ,b_k,\pm 2, -b_k,\ldots, -b_2,-b_1).$$
\end{proposition}

	\begin{figure}
\labellist
\small\hair 2pt
\pinlabel {$a_0$} at 43 18
\pinlabel {$a_1$} at 92 53
\pinlabel {$a_2$} at 146 18
\pinlabel {$a_3$} at 192 53
\pinlabel {$a_{2n-1}$} at 286 53
\pinlabel {$\ldots$} at 240 32
\endlabellist	
\includegraphics{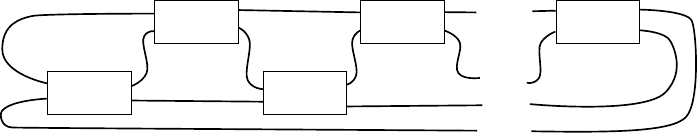}
	\caption{The two-bridge knot $C(a_0,\ldots,a_{2n-1})$.}
	\label{f:twobridge}
	\end{figure}

The following proposition of Burde-Zieschang tells us how to compute the relevant polynomial invariants:

\begin{proposition}[Burde-Zieschang {\cite[Proposition 12.23]{BZ}}]\label{p:conwayprop}
The Conway polynomial of a two-bridge knot expressed as $C(a_0,\ldots , a_{2n-1})$ has degree $\sum_{j=0}^{n-1} |a_{2i}|$.
\end{proposition}

\begin{remark}
The notation in Burde-Zieschang is different than that used in Proposition \ref{p:unknotprop}. To reconcile the conventions, the two-bridge knot diagram in Figure \ref{f:twobridge} can be converted to a 4-plat diagram as in Burde-Zieschang by pulling the inner strand on the right hand side of the figure over the outer strand. This has the effect of adding a new crossing (i.e. $a_{2n}=+1$) and adjusting $a_{2n-1}$ by $+1$.
\end{remark}

\begin{lemma}
There exists an infinite family of unknotting number one knots whose Alexander polynomials have arbitrarily high degree. 
\end{lemma}

\begin{proof}
Combining Propositions \ref{p:unknotprop} and \ref{p:conwayprop} shows that there exists an infinite family of two-bridge knots of unknotting number one such that the Conway polynomial has arbitrarily high degree. The lemma is thus immediate from the fact that the Conway polynomial is related to the Alexander polynomial by the formula $\nabla (t-t^{-1}) = \Delta (t^2)$.
\end{proof}

\begin{proof}[Proof of Theorem \ref{t:main}]
Let $\{ K_j \}$ be a sequence of knots of unknotting number 1 such that the degree of their Alexander polynomials goes to infinity, and let $\TTj$ be the associated (topologically trivial) tori from Section \ref{s:constructing}. 

Since there is a log-transform on $\TTj$ that gives $X_{K_j}$, we have that

\begin{align*}
\lim_{j\to \infty}  B(X,\TTj) \geq& \lim_{j\to \infty} \left( \begin{array}{ll}
\text{max divisibility of the difference} \\
\text{between any two basic classes of } X_{K_j}
\end{array} \right) \\
 \geq& \lim_{j\to \infty} 4deg(\Delta_{K_j}) = \infty. 
\end{align*}

The second inequality follows because the knot surgery formula, $$SW_{X_K} = SW_X \cdotp \Delta_K(2[T]),$$ allows us to determine the basic classes of $X_K$ from those of $X$. Specifically, if $\kappa$ is a basic class of $X$, and the degree of $\Delta_{K_j}$ is $n$, then $X_{K_j}$ has $\kappa + 2n[T]$ and $\kappa - 2n[T]$ as basic classes (among others), and the divisibility of the difference of this pair of basic classes is $4n$, which serves a a lower bound for $B(X,\TTj)$.
Therefore, after possibly passing to a subsequence of the $\{X_{K_j}\}$ there are an infinite number of the $\TTj$ that are smoothly distinguished by their $B$ invariant.
 \end{proof}

\end{document}